\documentclass[letterpaper,10pt,pdflatex,reqno]{amsart}
\usepackage{amsmath,times}
\usepackage{amsthm}
\usepackage{amssymb}
\usepackage{latexsym}
\usepackage{amscd}
\usepackage[english]{babel}
\usepackage[latin1]{inputenc}
\usepackage{graphicx}
\usepackage{color}
\usepackage{verbatim}
\newtheorem{theorem}{Theorem}[section]
\newtheorem{lemma}[theorem]{Lemma}
\newtheorem{proposition}[theorem]{Proposition}

\theoremstyle{definition}
\newtheorem{definition}[theorem]{Definition}

\newtheorem{remark}[theorem]{Remark}

\newtheorem{innercustomthm}{Theorem}
\newenvironment{customthm}[1]
  {\renewcommand\theinnercustomthm{#1}\innercustomthm}
  {\endinnercustomthm}

\begin{document} 

\title{Okounkov bodies and the K\"ahler geometry of projective manifolds}
\author{David Witt Nystr\"om}
\maketitle

\begin{abstract}
Given a projective manifold $X$ equipped with an ample line bundle $L$, we show how to embed certain torus-invariant domains $D \subseteq\mathbb{C}^n$ into $X$ so that the Euclidean K\"ahler form on $D$ extends to a K\"ahler form on X lying in the first Chern class of $L$. This is done using Okounkov bodies $\Delta(L)$, and the image of $D$ under the standard moment map will approximate $\Delta(L)$. This means that the volume of $D$ can be made to approximate the K\"ahler volume of $X$ arbitrarily well. As a special case we can let $D$ be an ellipsoid. We also have similar results when $L$ is just big.
\end{abstract}

\section{Introduction}

In toric geometry there is a beautiful correspondence between Delzant polytopes $\Delta$ and toric manifolds $X_{\Delta}$ equipped with an ample torus-invariant line bundles $L_{\Delta}$. This is important since many properties of $L_{\Delta}$ can be read directly from the polytope $\Delta.$ Okounkov found in \cite{Okounkov1,Okounkov2} a generalization of sorts, namely a way to associate a convex body $\Delta(L)$ to an ample line bundle $L$ on a projective manifold $X$, depending on the choice of a flag of smooth irreducible subvarieties in $X$. In the toric case, if one uses a torus-invariant flag, one essentially gets back the polytope $\Delta$. The convex bodies $\Delta(L)$ are now called Okounkov bodies. They were popularized by the work of Kaveh-Khovanskii \cite{Kaveh1,Kaveh2} and Lazarsfeld-Musta\c{t}\u{a} \cite{LazMus}, where it was shown that the construction works in far greater generality, e.g. big line bundles (for more references see the exposition \cite{BouckOk}).

Recall that the volume of a line bundle measures the asymptotic growth of $h^0(X,kL):=\dim_{\mathbb{C}}H^0(X,kL)$: $$\textrm{vol}(L):=\limsup_{k\to \infty}\frac{n!}{k^n}h^0(X,kL).$$ $L$ is then said to be big if $\textrm{vol}(L)>0$. When $L$ is ample or nef, asymptotic Riemann-Roch together with Kodaira vanishing shows that $\textrm{vol}(L)=(L^n)$. This is not true in general, since $(L^n)$ can be negative while the volume always is nonnegative. 

The key fact about Okounkov bodies is that they capture this volume: 
\begin{equation} \label{volOk}
\textrm{vol}(L)=n!\textrm{vol}(\Delta(L)).
\end{equation} 
Here the volume of the Okounkov body is calculated using the Lebesgue measure. This means that results from convex analysis, e.g. the Brunn-Minkowski inequality, can be applied to study the volume of line bundles. 

In the toric setting, a fruitful way of thinking of $\Delta$ is as the image of a moment map. There is a holomorphic $(\mathbb{C}^*)^n$-action on $X_{\Delta}$ which lifts to $L_{\Delta}$ and choosing an $(S^1)^n$-invariant K\"ahler form $\omega_{\Delta}\in c_1(L_{\Delta})$ gives rise to a  symplectic moment map $\mu_{\omega_{\Delta}}$ whose image can be identified with $\Delta$. 

Building on joint work with Harada \cite{HarKav}, Kaveh shows in the recent work \cite{Kaveh3} how Okounkov body data can be used to gain insight into the symplectic geometry of $(X,\omega)$, where $\omega$ is some K\"ahler form in $c_1(L)$ (it does not matter which K\"ahler form $\omega\in c_1(L)$ one uses since by Moser's trick all such K\"ahler manifolds are symplectomorphic). 

In short, Kaveh constructs symplectic embeddings $f_k:((\mathbb{C}^*)^n,\eta_k) \hookrightarrow (X,\omega)$ where $\eta_{k}$ are $(S^1)^n$-invariant K\"ahler forms that depend on data related to a certain nonstandard Okounkov body $\Delta(L)$ (i.e. the order on $\mathbb{N}^n$ used is not the lexicographic one). As $k$ tends to infinity the image of the corresponding moment map will fill up more and more of $\Delta(L)$, showing that the symplectic volume of $((\mathbb{C}^*)^n,\eta_k)$ approaches that of $(X,\omega)$. Just as in \cite{HarKav} the construction uses the gradient-Hamiltonian flow introduced by Ruan \cite{Ruan}, and is thus fundamentally symplectic in nature. 

\subsection{Main results}

We first introduce the following notion:

\begin{definition}
We say that a K\"ahler manifold $(Y,\eta)$ fits into $(X,L)$ if for every relatively compact open set $U\subseteq Y$ there is a holomorphic embedding $f$ of $U$ into $X$ so that $f_*\eta$ extends to some K\"ahler form $\omega$ on $X$ lying in $c_1(L)$. If $\dim_{\mathbb{C}}Y=\dim_{\mathbb{C}}X=n$ and $$\int_Y\eta^n=\int_X c_1(L)^n$$ we say that $(Y,\eta)$ fits perfectly into $(X,L)$.
\end{definition}

The special case of $(\mathbb{C}^n,\eta)$ ($\eta$ being some nonstandard $S^1$-invariant K\"ahler form) fitting into $(X,L)$ was considered in \cite{WN}.

Let $$\mu(z):=(|z_1|^2,...,|z_n|^2),$$ which we note is a moment map of $(\mathbb{C}^n,\omega_{st})$ with respect to the standard torus-action (here $\omega_{st}:=dd^c|z|^2$ denotes the standard Euclidean K\"ahler form on $\mathbb{C}^n$). 

Pick a complete flag $X_{\bullet}:= X=X_0 \supset X_{1} \supset X_{n-1}\supset X_n= \{p\}$ of smooth irreducible subvarieties. One can then define the associated Okounkov body $\Delta(L)$. We introduce the notion of the Okounkov domain $D(L)\subseteq \mathbb{C}^n$ which is a torus-invariant domain with the property that $$\Delta(L)^{\circ}\subseteq \mu(D(L))\subseteq \Delta(L),$$ (in general both inclusions are strict). We note that by (\ref{volOk}) 
\begin{equation} \label{voleq2}
\int_{D(L)}\omega_{st}^n=\int_X c_1(L)^n.
\end{equation}

\begin{customthm}{A} \label{mainthmOk}
We have that $(D(L),\omega_{st})$ fits perfectly into $(X,L)$. We can furthermore choose each embedding $f:U\to X$ ($U\subset D(L)$) so that $$f^{-1}(X_i)=\{z_1=...=z_i=0\}\cap U.$$ 
\end{customthm}

So on $(f(U),f_*{\omega_{st}})\subseteq (X,\omega)$ there is a torus-action with moment map $\mu\circ f^{-1}$ whose image approximates $\Delta(L)$ and we can choose $U$ so that $$\int_{f(U)}\omega^n\approx \int_X \omega^n.$$    

These results are still true even when using some nonstandard additive order on $\mathbb{N}^n$ to define the Okounkov body $\Delta(L)$. In particular when using the deglex order, which gives rise to the infinitesimal Okounkov bodies that appear in \cite{LazMus} and in the recent work of K\"uronya-Lozovanu \cite{KL15a, KL15b}. 

When $L$ is very ample there is a particular choice of flag $X_{\bullet}$ which makes $D(L)$ an ellipsiod, namely the ellipsiod $E(1,...,1,(L^n))$ defined by the inequality $$\sum_{i=1}^{n-1}|z_i|^2+(L^n)^{-1}|z_n|^2<1.$$ This leads to the following theorem. 

\begin{customthm}{B} \label{thmB}
If $L$ is very ample then we have that $(E(1,...,1,(L^n)),0,\omega_{st})$ fits perfectly into $(X,L)$, and the associated embeddings can be chosen to be centered at any point $p\in X$.
\end{customthm}

There is an interesting connection between this result and the notion of Seshadri constants.

Recall the definition of the Seshadri constant $\epsilon(X,L,p),$ introduced by Demailly \cite{Demailly}.

\begin{definition}
The Seshadri constant of an ample line bundle $L$ at a point $p$ is given by $$\epsilon(X,L,p):=\inf_C \frac{L\cdot C}{\textrm{mult}_p C},$$ where the infimum is taken over all curves $C$ in $X.$
\end{definition}

One can show that the Seshadri constant $\epsilon(X,L,p)$ also measures the maximal size of embedded balls centered at $p$ such that the restricted K\"ahler structure is standard.

\begin{theorem} \label{Seshthm}
We have that $\epsilon(X,L,p)$ is equal to the supremum of $r$ such that $(B_r,0,\omega_{st})$ fits into $(X,L)$ with the embeddings centered at $p$.
\end{theorem} 

This result can be extracted from Lazarsfeld \cite{Lazarsfeld} (see Theorem 5.1.22 and Proposition 5.3.17); the main argument is due to McDuff-Polterovic \cite{McDuff}.

From Theorem \ref{Seshthm} follows the inequality $$\epsilon(X,L,p)\leq (L^n)^{\frac{1}{n}}.$$ When this inequality is strict for $(X,L,p)$ (which is the general case) it means that no ball $(B_r,0,\omega_{st})$ can fit perfectly into $(X,\omega_L)$ centered at $p$. Nevertheless Theorem \ref{thmB} says that one always can find an ellipsoid which fits perfectly into $(X,\omega_L)$ centered at $p$.

Let $\Delta(L)$ be an infinitesimal Okounkov body at $p$ and $D(L)$ the corresponding Okounkov domain. One can easily show that $$\epsilon(X,L,p)=\sup\{r: B_r\subseteq D(L)\}$$ so thus Theorem \ref{mainthmOk} can be thought of as strengthening of Theorem \ref{Seshthm}.

The proof of Theorem \ref{mainthmOk} relies on finding suitable toric degenerations. Here we follow \cite{Anderson}, but as in \cite{Ito} and \cite{Kaveh3} we do not degenerate the whole section ring $R(L)$ but rather $H^0(X,kL)$ for fixed $k$. We couple the degeneration with a max construction to find a suitable positive hermitian metric of $L$, whose curvature form will provide the appropriate K\"ahler form $\omega$ in the theorem. We recently used this technique to construct K\"ahler embeddings related to canonical growth conditions \cite[Thm. C]{WN}.

\subsection{The big case}

We have similar results when $L$ is just big. Then there are no longer any K\"ahler forms in $c_1(L)$ so instead we use K\"ahler currents in $c_1(L)$ with analytic singularities.

\begin{definition}
If $L$ is big we say that a K\"ahler manifold $(Y,\eta)$ fits into $(X,L)$ if for every relatively compact open set $U\subseteq Y$ there is a holomorphic embedding $f$ of $U$ into $X$ such that $f_*\eta$ extends to a K\"ahler current with analytic singularities on $X$ lying in $c_1(L)$. If $\dim_{\mathbb{C}}Y=\dim_{\mathbb{C}}X=n$ and $$\int_Y\eta^n=\int_X c_1(L)^n$$ we say that $(Y,\eta)$ fits perfectly into $(X,L)$
\end{definition}

\begin{customthm}{C} \label{mainthmbig}
We have that $(D(L),\omega_{st})$ fits perfectly into $(X,L)$. We can furthermore choose each embedding $f:U\to X$ ($U\subset D(L)$) so that $$f^{-1}(X_i)=\{z_1=...=z_i=0\}\cap U.$$ 
\end{customthm}

\subsection{Related work}

The work of Kaveh \cite{Kaveh3} which inspired this paper has already been mentioned. This built on joint work with Harada \cite{HarKav}, which in turn used the work of Anderson \cite{Anderson} on toric degenerations.

Anderson showed in \cite{Anderson} how, given some assumptions, the data generating the Okounkov body also gives rise to a degeneration of $(X,L)$ into a possibly singular toric variety $(X_{\Delta},L_{\Delta})$, where $\Delta=\Delta(L)$ (the assumptions force $\Delta(L)$ to be a polytope, which is not the case in general). In their important work \cite{HarKav} Harada-Kaveh used this to, under the same assumptions, to construct a completely integrable system $\{H_i\}$ on $(X,\omega)$, with $\omega$ a K\"ahler form in $c_1(L)$, such that $\Delta(L)$ precisely is the image of the moment map $\mu:=(H_1,...,H_n)$. More precisely, they find an open dense subset $U$ and a Hamiltonian $(S^1)^n$-action on $(U,\omega)$ such that the corresponding moment map $\mu:=(H_1,...,H_n)$ extends continuosly to the whole of $X$. Their construction uses the gradient-Hamiltonian flow introduced by Ruan \cite{Ruan}.

In the recent work \cite{WN}, given an ample line bundle $L$ and a point $p\in X$, we show how to construct an $(S^1)$-invariant plurisubharmonic function $\phi_{L,p}$ on $T_p X$, such that the corresponding growth condition $\phi_{L,p}+O(1)$ is canonically defined. We then prove that the growth condition provides a sufficient condition for certain K\"ahler balls $(B_1,\eta)$ to be embeddable into some $(X,\omega)$ with $\omega\in c_1(L)$ and K\"ahler \cite[Thm. D]{WN}.

The very general Seshadri constant $\epsilon(X,L;1)$ is defined as the supremum of $\epsilon(X,L;p)$ over $X$, which is the same as the Seshadri constant at a very general point. In \cite{Ito} Ito proved that if $\Delta$ is an integer polytope such that $\frac{1}{k}\Delta\subset \Delta(L)$ then $$\epsilon(X,L;1)\geq \frac{1}{k}\epsilon(X_{\Delta},L_{\Delta};1).$$ He did this using the same kind of toric degeneration as was later used by Kaveh in \cite{Kaveh3} and that we use here. One can easily show that this also follows from our results. This illustrates the difference between our results and those of Kaveh in \cite{Kaveh3}. Since Kaveh's construction is symplectic that only implies the weaker symplectic version of Ito's theorem, namely the corresponding lower bound on the Gromov width \cite[Cor. 8.4]{Kaveh3}.  

\subsection{Acknowledgements}

We wish to thank Julius Ross for many fruitful discussions relating to the topic of this paper. We also thank Kiumars Kaveh for sharing his very interesting preprint \cite{Kaveh3}.

During the preparation of this paper the author has received funding from the People Programme (Marie Curie Actions) of the European Union's Seventh Framework Programme (FP7/2007-2013) under REA grant agreement no 329070.

\section{Okounkov bodies and domains} \label{Secok}

Let $L$ be a big line bundle on a projective manifold $X$. Choose a complete flag $X=X_0 \supset X_{1} \supset X_{n-1}\supset X_n= \{p\}$ of smooth irreducible subvarieties of $X$, $\textrm{codim}X_i=i$. We can then choose local holomorphic coordinates $z_i$ centered at $p$ such that in some neighbourhood $U$ of $p$, $$X_i\cap U=\{z_1=...=z_i=0\}\cap U.$$ Also pick a local trivialization of $L$ near $p$. Locally near $p$ we can then write any section $s\in H^0(X,kL)$ as a Taylor series $$s=\sum_{\alpha}a_{\alpha}z^{\alpha}.$$ When $s$ is nonzero we let $$v(s):=\min\{\alpha: a_{\alpha}\neq 0\},$$ where the mininum is taken with respect to the lexicographic order (or some other additive order of choice). The Okounkov body $\Delta(L)$ of $L$ (for ease of notation the dependence of the flag is usually not written out) is then defined as $$\Delta(L):=Conv\left(\left\{\frac{v(s)}{k}: s\in H^0(X,kL)\setminus \{0\}, k\geq 1\right\}\right).$$ Here $Conv$ means the closed convex hull.

\begin{remark}
Another natural choice of order on $\mathbb{N}^n$ to use is the deglex order. This means that $\alpha<\beta$ if $|\alpha|<|\beta|$ ($|\alpha|:=\sum_i \alpha_i$) or else if $|\alpha|=|\beta|$ and $\alpha$ is less than $\beta$ lexicographically. If one uses this order to define the Okounkov body, this will only depend on the flag of subspaces of $T_p X$ given by $T_p X_i$, and it will be equivalent to the infinitesimal Okounkov body considered in \cite{LazMus} and in the recent work of K\"uronya-Lozovanu \cite{KL15a,KL15b} (see \cite{WN}).
\end{remark}

Let us define $$\mathcal{A}(kL):=\{v(s): s\in H^0(X,kL)\setminus \{0\}\}.$$ By elimination we can find sections $s_{\alpha}\in H^0(X,kL)$, $\alpha\in \mathcal{A}(kL)$, such that $$s_{\alpha}=z^{\alpha}+\sum_{\beta>\alpha, \beta\notin \mathcal{A}(kL)}a_{\beta}z^{\beta}.$$ If $$s=\sum_{\alpha\in \mathcal{A}(kL)}a_{\alpha}z^{\alpha}+\sum_{\beta \notin \mathcal{A}(kL)}a_{\beta}z^{\beta}$$ then we must have that $$s=\sum_{\alpha\in \mathcal{A}(kL)}a_{\alpha}s_{\alpha},$$ because otherwise we would have that $v(s-\sum a_{\alpha}s_{\alpha})\notin \mathcal{A}(kL)$. It follows that $s_{\alpha}$ is a basis for $H^0(X,kL)$  so  
\begin{equation} \label{numberofp}
|\mathcal{A}(kL)|=h^0(X,kL),
\end{equation}
where $|\mathcal{A}(kL)|$ denotes the number of points in $\mathcal{A}(kL)$. 

If $s=z^{\alpha_1}+\sum_{\beta>\alpha_1}a_{\beta}z^{\beta}$ and $t=z^{\alpha_2}+\sum_{\beta>\alpha_2}b_{\beta}z^{\beta}$ then $$st=z^{\alpha_1+\alpha_2}+\sum_{\beta>\alpha_1+\alpha_2}c_{\beta}z^{\beta}$$ and hence $v(st)=v(s)+v(t)$. This implies that for $k,m\in \mathbb{N}:$ 
\begin{equation} \label{semigroup}
\mathcal{A}(kL)+\mathcal{A}(mL)\subseteq \mathcal{A}((k+m)L)
\end{equation} 
and thus $$\Gamma(L):=\bigcup_{k\geq 1}\mathcal{A}(kL)\times \{k\}\subseteq \mathbb{N}^{n+1}$$ is a semigroup. 

Combined with a result by Khovanskii \cite[Prop. 2]{Khovanskii} it leads to the proof of the key result (see e.g. \cite{Kaveh1,Kaveh2} or \cite{LazMus}).

\begin{theorem} \label{Okvolume}
We have that $$\textrm{vol}(L)=n!\textrm{vol}(\Delta(L)),$$ where the volume of $\Delta(L)$ is calculated using the Lebesgue measure. 
\end{theorem}

From this we see that when $X$ has dimension one, $\Delta(L)$ is an interval of lenght $deg(L)$. When $L$ is ample one gets that $0\in \Delta(L)$ and thus 
\begin{equation} \label{okcurve}
\Delta(L)=[0,deg(L)].
\end{equation}

Let $$\Delta_k(L):=\frac{1}{k}Conv(\mathcal{A}(kL)).$$ From (\ref{semigroup}) we see that for $k,m\in \mathbb{N}:$
\begin{equation} \label{okincl}
\Delta_k(L)\subseteq \Delta_{km}(L).
\end{equation}

The following lemma is also an immediate consequence of the result of Khovanskii (see e.g. \cite[Lem. 2.3]{DWN}).

\begin{lemma} \label{compincl}
Let $K$ be a compact subset of $\Delta(L)^{\circ}$. Then for $k>0$ divible enough we have that $$K\subset \Delta_k(L).$$
\end{lemma}

From this it follows that $$\Delta(L)^{\circ}=\bigcup_{k\geq 1}\Delta_k(L)^{\circ}.$$ Let $\Delta_k(L)^{ess}$ denote the interior of $\Delta_k(L)$ as a subset of $\mathbb{R}^n_{\geq 0}$ with its induced topology.

\begin{definition}
We define the essential Okounkov body $\Delta(L)^{ess}$ as $$\Delta(L)^{ess}:=\bigcup_{k\geq 1}\Delta_k(L)^{ess}.$$
\end{definition} 

By (\ref{okincl}) we get that for any $k,m\in \mathbb{N},$ $\Delta_k(L)^{ess}\subseteq \Delta_{km}(L)^{ess}$ and thus $$\Delta(L)^{ess}=\bigcup_{k\geq 1}\Delta_{k!}(L)^{ess}.$$ We also see that $\Delta_{k!}(L)^{ess}$ is increasing in $k$ which then implies that $\Delta(L)^{ess}$ is an open convex subset of $\mathbb{R}_{\geq 0}^n$.  

\begin{lemma} \label{compincless}
Let $K$ be a compact subset of $\Delta(L)^{ess}$. Then for $k>0$ divible enough we have that $$K\subset \Delta_k(L)^{ess}.$$
\end{lemma}

This is proved in the same way as Lemma \ref{compincl}.

It is easy to see that $$\Delta(L)\cap \{x_1=0\}\subseteq \Delta(L_{|X_1}),$$ where $\Delta(L_{|X_1})$ is defined using the induced flag $X_1\supset X_2 \supset ...\supset X_n$. When $L$ is ample one can use Ohsawa-Takegoshi to prove that we have an equality
\begin{equation} \label{okrestr}
\Delta(L)\cap \{x_1=0\}=\Delta(L_{|Y_1}),
\end{equation} 
(see e.g. \cite{DWN}).

Let $L_1$ denote the holomorphic line bundle associated with the divisor $X_1$. An important fact, proved by Lazarsfeld-Musta\c{t}\u{a} in \cite{LazMus} is that 
\begin{equation} \label{oktrans}
\Delta(L)\cap\{x_1\geq r\}=\Delta(L-rL_1)+re_1.
\end{equation}  

For $a\in \mathbb{R}^n$ we let $\Sigma_a$ denote the convex hull of $\{0,a_1e_1, a_2e_2,..., a_ne_n\}$ and $\Sigma_a^{ess}$ the interior of $\Sigma_a$ as a subset of $\mathbb{R}_{\geq 0}^n$. 

\begin{proposition} \label{specialbody}
If $L$ is very ample then there is a flag $X=X_0 \supset X_{1} \supset ...\supset X_n= \{p\}$ of smooth irreducible subvarieties of $X$ such that $$\Delta(L)=\Sigma_{(1,...,1,(L^n))}$$ and $$\Delta(L)^{ess}=\Sigma_{(1,...,1,(L^n))}^{ess}.$$
\end{proposition}

\begin{proof}
Since $L$ is very ample we can find a flag $X=X_0 \supset X_{1} \supset ...\supset Y_n= \{p\}$ of smooth irreducible subvarieties of $X$ such that for each $i\in \{1,...,n\}$ the line bundle $L_{|X_{i-1}}$ is associated with the divisor $X_i$ in $X_{i-1}$. 

From repeated use of (\ref{okrestr}) and (\ref{oktrans}) we get that $$\Delta(L)\cap \{x_1=r_1,...,x_{n-1}=r_{n-1}\}=\Delta((1-\sum_i r_i)L_{|X_{n-1}})=[0,((1-\sum_i r_i))(L^n)],$$ using (\ref{okcurve}) and the fact that $deg(L_{X_{n-1}})=(L^n)$. In other words $$\Delta(L)=\Sigma_{(1,...,1,(L^n))}.$$ Since $$\Delta(L_{|Y_{n-1}})^{ess}=[0,(L^n))$$ we similarly get that $$\Delta(L)^{ess}=\Sigma_{(1,...,1,(L^n))}^{ess}.$$ 
\end{proof}

Recall that $$\mu(z):=(|z_1|^2,...,|z_n|^2).$$

\begin{definition}
We define the Okounkov domain $D(L)$ to be $$D(L):=\mu^{-1}(\Delta(L)^{ess}).$$
\end{definition} 

We note that $D(L)$ is a bounded domain in $\mathbb{C}^n$. We also note that when $\Delta(L)^{ess}=\Sigma_{(1,...,1,(L^n))}^{ess}$ we get that $D(L)=E(1,...,1,(L^n))$, i.e. the ellipsoid defined by the inequality $$\sum_{i=1}^{n-1}|z_i|^2+(L^n)^{-1}|z_n|^2<1.$$

\section{Torus-invariant K\"ahler forms and moment maps} \label{moment}

Let $(M,\omega)$ be a symplectic manifold. Assume that there is an $S^1$-action on $M$ which preserves $\omega$ and let $V$ be the generating vector field. We must have that $\mathcal{L}_V\omega=0$. By Cartan's formula we have that $$d(\omega(V,\cdot))=\mathcal{L}_V\omega-d\omega(V,\cdot)=0,$$ so the one-form $\omega(V,\cdot)$ is closed. A function $H$ is called a Hamiltonian for the $S^1$-action if $$dH=\omega(V,\cdot).$$ If $H$ is a Hamiltonian then clearly so is $H+c$ for any constant $c$. If $M$ has an $(S^1)^n$-action which preserves $\omega$, and each individual $S^1$-action has a Hamiltonian $H_i$, we call the map $\mu:=(H_1,...,H_n)$ a moment map for the $(S^1)^n$-action. There is a more invariant way of defining the moment map so that it takes values in the dual of the Lie algebra of the acting group, but we will not go into that here.

Let $\mathcal{A}\subseteq \mathbb{N}^n$ be a finite set and assume that $Conv(\mathcal{A})^{ess}$ is nonempty. Let $$D_{\mathcal{A}}:=\mu^{-1}(Conv(\mathcal{A})^{ess})=\mu^{-1}(Conv(\mathcal{A}))^{\circ}$$ and let $X_{\mathcal{A}}$ denote the manifold we get by removing from $\mathbb{C}^n$ all the submanifolds of the form $\{z_{i_1}=...=z_{i_k}=0\}$ which do not intersect $D_{\mathcal{A}}$. Then $$\phi_{\mathcal{A}}:=\ln\left(\sum_{\alpha\in \mathcal{A}}|z^{\alpha}|^2\right)$$ is a smooth strictly psh function on $X_{\mathcal{A}}$ and we denote by $\omega_{\mathcal{A}}:=dd^c\phi_{\mathcal{A}}$ the corresponding K\"ahler form. 

Note that we can write $$\phi_{\mathcal{A}}(z)=u_{\mathcal{A}}(x):=\ln\left(\sum_{\alpha\in \mathcal{A}}e^{x\cdot \alpha}\right),$$ where $x_i:=\ln|z_i|^2$ and $u_{\mathcal{A}}$ is a convex function on $\mathbb{R}^n$.

Let us think of $(X_{\mathcal{A}},\omega_{\mathcal{A}})$ as a symplectic manifold. The symplectic form $\omega_A$ is clearly invariant under the standard $(S^1)^n$-action on $X_{\mathcal{A}}$ and it is a classical fact that $\mu_{\mathcal{A}}: z\mapsto \nabla u(x)$ is a moment map for this action. To see this we define $u_{\mathcal{A}}(w):=u_{\mathcal{A}}(\textrm{Re} w)$ for $w\in X_{\mathcal{A}}$ and note that $u_{\mathcal{A}}$ is the pullback of $\phi_{\mathcal{A}}$ by the holomorphic map $f:w\to e^{w/2}$. We then have that $f^*\omega_{\alpha}=dd^c u_{\alpha}$. The pullback of the vector field generating the $i$:th $S^1$-action is $(2\pi)\partial/\partial x_i$, so to show that $\partial/\partial x_i u_{\alpha}$ is a Hamiltonian we need to establish that $$d\frac{\partial}{\partial x_i}u_{\mathcal{A}}=dd^c u_{\mathcal{A}}((2\pi)\partial/\partial x_i,\cdot).$$ This is easily checked using that $$dd^cu_{\mathcal{A}}=\frac{1}{2\pi i}\sum_{i,j}\frac{\partial^2 u}{\partial x_i \partial x_j}dw_i\wedge d\bar{w}_j.$$ 

Clearly $$\mu_{\mathcal{A}}(\mathbb{C^*})^n=Conv(\mathcal{A})^{\circ}$$ while $$\mu_{\mathcal{A}}(X_{\mathcal{A}})=Conv(\mathcal{A})^{ess}.$$

Another classical fact is that for any open $(S^1)^n$-invariant set $U\subseteq X_{\mathcal{A}}$ we have that $$\int_U \omega^n_{\mathcal{A}}=\textrm{vol}(\mu_{\mathcal{A}}(U)).$$ To see this, write $f^{-1}(U)=V\times (i\mathbb{R})^n$ and thus $$\int_U \omega^n_{\mathcal{A}}=\int_{V\times (i[0,2\pi])^n}(dd^c u_{\mathcal{A}})^n=\int_V \textrm{det}(\textrm{Hess}(u))=\int_{\nabla u(V)}dx,$$ where $\textrm{Hess}(u)$ denotes the Hessian of $u$, and in the last step we used that this is equal to the Jacobian of $\nabla u$.  

\begin{lemma} \label{embdomain}
Let $U$ be a relatively compact open subset of $D_{\mathcal{A}}$. Then there exists a smooth function $g:X_{\mathcal{A}}\to \mathbb{R}$ with compact support such that $\omega:=\omega_{\mathcal{A}}+dd^c g$ is K\"ahler and on $U$ we have that $\omega=\omega_{st}$. 
\end{lemma}

\begin{proof}
Using Legendre transforms one can find a smooth $(S^1)^n$-invariant strictly psh function $\phi$ on $X_{\mathcal{A}}$ which is equal to $|z|^2$ on $U$ and such that the image of the gradient of $u(x):=\phi(e^{x_1/2},...,e^{x_n/2})$ is compactly supported in $Conv(\mathcal{A})^{ess}$. One sees then that $\phi_{\mathcal{A}}-\phi$ is proper on $X_{\mathcal{A}}$. Let $C$ be a constant such that $\phi+C>\phi_{\mathcal{A}}$ on $D$. Pick some $\delta>0$ and let $\max_{reg}(x,y)$ be a smooth convex function such that $\max_{reg}(x,y)=\max(x,y)$ whenever $|x-y|>\delta$. Then $\phi':=\max_{reg}(\phi+C+\delta,\phi_{\mathcal{A}})$ is a smooth strictly psh function on $X_{\mathcal{A}}$ which is equal to $\phi+C+\delta$ on $U$ while being equal to $\phi_{\mathcal{A}}$ outside of some compact set. It follows that $g:=\phi'-\phi_{\mathcal{A}}$ has the desired properties.
\end{proof}

\section{K\"ahler embeddings of domains}

In the introduction we had the following definition.

\begin{definition}
We say that a K\"ahler manifold $(Y,\eta)$ fits into $(X,L)$ if for every relatively compact open set $U\Subset Y$ there is a holomorphic embedding $f$ of $U$ into $X$ such that $f_*\eta$ extends to a K\"ahler form on $X$ lying in $c_1(L)$. If in addition $$\int_Y \eta^n=\int_X c_1(L)^n$$ then we say that $(Y,\eta)$ fits perfectly into $(X,L)$.
\end{definition}

Recall that $\mathcal{A}(kL):=\{v(s):s\in H^0(X,kL)\}$.  

\begin{theorem} \label{mainthm2}
Assume that $L$ ample. Then for $k$ large enough, $(X_{\mathcal{A}(kL)},\omega_{\mathcal{A}(kL)})$ fits into $(X,kL)$, and each associated K\"ahler embedding $f:U\to X$ can be chosen so that $$f^{-1}(X_i)=\{z_1=...=z_i=0\}\cap U.$$
\end{theorem}

Before proving Theorem \ref{mainthm2} we need a simple lemma. 

\begin{lemma} \label{vector}
For any finite set $\mathcal{A}\subseteq \mathbb{N}^n$ there exists a $\gamma\in (\mathbb{N}_{>0})^n$ such that for all $\alpha\in \mathcal{A}$: 
\begin{equation} \label{desprop}
\alpha<\beta \in \mathbb{N}^n \Longrightarrow \alpha \cdot \gamma < \beta \cdot \gamma.
\end{equation}
\end{lemma}

This is a standard fact which is true for any additive order, see e.g. \cite[Lem. 8]{Anderson}. It plays a key role in constructing toric degenerations.

\begin{proof}
Pick a number $C\in \mathbb{N}$ such that $C>|\alpha|$ for all $\alpha\in \mathcal{A}$. We claim that $$\gamma:=\sum_{i}(2C)^{n-i}e_i$$ has the desired property (\ref{desprop}). Assume that $\alpha<\beta$. By definition there is an index $j$ such that $\alpha_i=\beta_i$ for $i<j$ while $\beta_j>\alpha_j$. It follows that 
\begin{eqnarray*}
(\beta-\alpha)\cdot \gamma=\sum_i (2C)^{n-i}(\beta_i-\alpha_i)=(2C)^{n-j}(\beta_j-\alpha_j)+\sum_{i>j}(2C)^{n-i}(\beta_i-\alpha_i)\geq \\ \geq (2C)^{n-j}-|\alpha|\sum_{i>j}(2C)^{n-i}\geq C^{n-j}>0.
\end{eqnarray*}
\end{proof}

We can now prove Theorem \ref{mainthm2}. As in \cite{Kaveh3} the proof relies on a toric deformation, given by a suitable choice of $\gamma$. However, instead of coupling it with a gradiant-Hamiltonian flow, we finish the proof using a max construction. This is similar to the proof of Theorem D in \cite{WN}.

\begin{proof}
Recall that we have local holomorphic coordinates $z_i$ centered at $p$. We assume that the unit ball $B_1\subset \mathbb{C}^n$ lies in the image of the coordinate chart $z: V \to \mathbb{C}^n$.  

Let $k$ be large enough so that $Conv(\mathcal{A}(kL))$ has nonempty interior and let $U$ be a relatively compact open set in $X_{\mathcal{A}(kL)}$.

Pick a basis $s_{\alpha}$ for $H^0(X,kL)$ indexed by $\mathcal{A}(kL)$ such that locally $$s_{\alpha}=z^{\alpha}+\sum_{\beta>\alpha} a_{\beta}z^{\beta}.$$ Pick a $\gamma$ as in Lemma \ref{vector} with $\mathcal{A}:=\mathcal{A}(kL)$ and let $\tau^{\gamma}z:=(\tau^{\gamma_1}z_1,...,\tau^{\gamma_n}z_n)$. It follows that 
\begin{equation} \label{approxs}
s_{\alpha}(\tau^{\gamma}z)=\tau^{\alpha\cdot \gamma}(z^{\alpha}+o(|\tau|))
\end{equation} 
for $\tau^{\gamma}z\in B_1$.

Let $f:X_{\mathcal{A}(kL)}\to [0,1]$ be a smooth function such that $f\equiv 0$ on $U$ and $f\equiv 1$ on the complement of some smoothly bounded compact set $K\subseteq X_{\mathcal{A}(kL)}$. Pick $0<\delta \ll 1$ such that $$\phi:=\phi_{\mathcal{A}(kL)}-4\delta f$$ is still strictly psh. It follows from (\ref{approxs}) that we can pick $0<\tau\ll 1$ such that $\tau^{\gamma}z\in B_1$ whenever $z\in K$ and so that $$\phi>\ln\left(\sum_{\alpha\in \mathcal{A}(kL)}\left|\frac{s_{\alpha}(\tau^{\gamma}z)}{\tau^{\alpha\cdot \gamma}}\right|^2\right)-\delta$$ on $U$ while $$\phi<\ln\left(\sum_{\alpha\in \mathcal{A}(kL)}\left|\frac{s_{\alpha}(\tau^{\gamma}z)}{\tau^{\alpha\cdot \gamma}}\right|^2\right)-3\delta$$ near $\partial K$. 

Let $\max_{reg}(x,y)$ be a smooth convex function such that $\max_{reg}(x,y)=\max(x,y)$ whenever $|x-y|>\delta$. Then the regularized maximum $$\phi':=\max_{reg}\left(\phi,\ln\left(\sum_{\alpha\in \mathcal{A}(kL)}|\frac{s_{\alpha}(\tau^{\gamma}z)}{\tau^{\alpha\cdot \gamma}}|^2\right)-2\delta\right)$$ is smooth and strictly plurisubharmonic on $X_{\mathcal{A}(kL)}$, identically equal to $\phi$ on $U$ while identically equal to $\ln(\sum_{\alpha\in \mathcal{A}(kL)}|\frac{s_{\alpha}(\tau^{\gamma}z)}{\tau^{\alpha\cdot \gamma}}|^2)-2\delta$ near the boundary of $K$. We get that $$\omega:=dd^c\phi'$$ is equal to $\omega_{\mathcal{A}(kL)}$ on $U$. 

If we assume that $k$ is large enough so that $kL$ is very ample then $\ln(\sum_{\alpha\in \mathcal{A}(kL)}|\frac{s_{\alpha}(\tau^{\gamma}z)}{\tau^{\alpha\cdot \gamma}}|^2)$ extends as a positive metric of $kL$ and thus $\omega$ extends to a K\"ahler form in $c_1(kL)$. 

Since $U$ was arbitrary this shows that $(X_{\mathcal{A}(kL)},\omega_{\mathcal{A}(kL)})$ fits into $(X,kL)$. We also note that the embedding $f$ of $U$ into $X$ was given by $z\mapsto \tau^{\gamma}z$, and thus we have that $$f_R^{-1}(X_i)=\{z_1=...=z_i=0\}\cap U.$$  
\end{proof}

We can now combine this result with Lemma \ref{embdomain} to obtain Theorem \ref{mainthmOk}.

\begin{customthm}{A} 
We have that $(D(L),\omega_{st})$ fits perfectly into $(X,L)$ and each associated K\"ahler embedding $f:U\to X$ can be chosen so that $$f^{-1}(X_i)=\{z_1=...=z_i=0\}\cap U.$$  
\end{customthm}

\begin{proof}
If $U$ is a relatively compact open set in $D(L)$ then by Lemma \ref{compincless} for $k>0$ divisible enough the closure of $U$ is contained in $\mu^{-1}(\Delta_k(L)^{ess})$, or in other words, $\sqrt{k}U$ is relatively compact in $D_{\mathcal{A}(kL)}$, which is in turn relatively comapact in $X_{\mathcal{A}(kL)}$. Thus by Lemma \ref{embdomain} there exists a smooth function $g:X_{\mathcal{A}(kL)}\to \mathbb{R}$ with support on a relatively compact set $U'$ such that $\omega:=\omega_{\mathcal{A}(kL)}+dd^c g$ is K\"ahler and on $\sqrt{k}U$ we have that $\omega=\omega_{st}$. By Theorem \ref{mainthm2}, if $k$ is large enough, $(X_{\mathcal{A}(kL)},\omega_{\mathcal{A}(kL)})$ fits into $(X,kL)$. Thus we can find a holomorphic embedding $f':U'\to X$ such that $f'_*\omega_{\mathcal{A}(kL)}$ extends to a K\"ahler form $\omega\in c_1(kL)$. Then letting $f:U\to X$ be defined as $f(z):=f'(\sqrt{k}z)$ we get that $f_*\omega_{st}=\frac{1}{k}f'_*{\omega_{st}}_{|\sqrt{k}U}$ extends to a K\"ahler form $\omega\in c_1(L)$.  

That $$\int_{D(L)}\omega_{st}^n=\int_X c_1(L)^n$$ followed from Theorem \ref{Okvolume} and it is clear that the $f:U\to X$ we found had the property that $$f^{-1}(X_i)=\{z_1=...=z_i=0\}\cap U.$$
\end{proof}

\begin{customthm}{B}
If $L$ is very ample then we have that $(E(1,...,1,(L^n)),0,\omega_{st})$ fits perfectly into $(X,L)$, and the associated embeddings can be chosen to be centered at any point $p\in X$.
\end{customthm}

\begin{proof}
This follows directly from combining Theorem \ref{mainthmOk} with Proposition \ref{specialbody}
\end{proof}

\section{Big line bundles}

If $L$ is big but not ample there are no K\"ahler forms in $c_1(L)$. Instead one can consider K\"ahler currents with analytic singularities that lies in $c_1(L)$. We can use these to define what it should mean for a K\"ahler manifold $(Y,\eta)$ to fit into $(X,L)$ when $L$ is just big.   

\begin{definition}
If $L$ is big we say that a K\"ahler manifold $(Y,\eta)$ fits into $(X,L)$ if for every relatively compact open set $U\subseteq Y$ there is a holomorphic embedding $f$ of $U$ into $X$ such that $f_*\eta$ extends to a K\"ahler current with analytic singularities on $X$ lying in $c_1(L)$. If $\dim_{\mathbb{C}}Y=\dim_{\mathbb{C}}X=n$ and $$\int_Y\eta^n=\int_X c_1(L)^n$$ we say that $(Y,\eta)$ fits perfectly into $(X,L)$
\end{definition}

\begin{customthm}{C}
We have that $(D(L),\omega_{st})$ fits perfectly into $(X,L)$. We can furthermore choose each embedding $f:U\to X$ ($U\subset D(L)$) so that $$f^{-1}(X_i)=\{z_1=...=z_i=0\}\cap U.$$ 
\end{customthm}

If $L$ is big and $k$ is large enough, then if $s_m$ is a basis for $H^0(kL)$ we get that \\$dd^c\ln(\sum_m |s_m|^2)$ is a K\"ahler current with analytical singularities which lies in $c_1(kL)$. Thus one proves Theorem \ref{mainthmbig} exactly as in the ample case.

\noindent {\sc David Witt Nystr\"om, 
Department of Mathematics, Chalmers University of Technology, Sweden \\wittnyst@chalmers.se, danspolitik@gmail.com

\end{document}